\documentclass{article}
\usepackage{amsmath,amssymb,amsthm,amsfonts,float,graphicx,url,nicefrac}
\usepackage[capitalise]{cleveref}
\usepackage[paper=letterpaper,margin=1in]{geometry}

\newtheorem{thm}{Theorem}[section]
\newtheorem{cor}[thm]{Corollary}
\newtheorem{lem}[thm]{Lemma}
\newtheorem{prop}[thm]{Proposition}
\newtheorem{obs}[thm]{Observation}
\newtheorem{fact}[]{Fact}
\newtheorem{defn}[thm]{Definition}
\newtheorem{ex}[thm]{Example}

\title{Power domination with random sensor failure}
\author{Beth Bjorkman\thanks{US Air Force Research Lab, Wright-Patterson AFB, OH (beth.morrison@us.af.mil)} \and Zachary Brennan\thanks{Iowa State University, Ames, IA (brennanz@iastate.edu)} \and Mary Flagg\thanks{University of St. Thomas, Houston, TX 77006 (flaggm@stthom.edu)} \and Johnathan Koch\thanks{Applied Research Solutions, Beavercreek, OH (jkoch@appliedrs.com)}}

\begin{document}
    \maketitle
    \begin{abstract}
        The power domination problem seeks to determine the minimum number of phasor measurement units (PMUs) needed to monitor an electric power network.
        We introduce random sensor failure before the power domination process occurs and call this the \emph{fragile power domination process}.
        For a given graph, PMU placement, and probability of PMU failure $q$, we study the expected number of observed vertices at the termination of the fragile power domination process.
        This expected value is a polynomial in $q$, which we relate to fault-tolerant and PMU-defect-robust power domination.
        We also study the probability of that the entire graph becomes observed and give results for some graph families.
    \end{abstract}
    \section{Introduction}
        The power domination problem, defined by Haynes et al. in \cite{hhhh02}, seeks to find the placement of the minimum number of phasor measurement units (PMUs) needed to monitor an electric power network.
        This process was simplified by Brueni and Heath in \cite{bh05}.
        PMUs are placed on an initial set of vertices of a graph, and vertex observation rules define a propagation process on the graph.
        
        Pai, Chang, and Wang \cite{pcw10} provided a variation to power domination which looked for the minimum number of PMUs needed to monitor a power network in the case where $k$ of the PMUs will fail, called \emph{$k$-fault-tolerant power domination}.
        Their work allows the placement of only one PMU per vertex.
        A generalization of this problem is \emph{PMU-defect-robust power domination} which allows for multiple PMUs to be placed at a given vertex, defined  by Bjorkman, Conrad, and Flagg \cite{bcf23}. 
        
        Fault-tolerant and PMU-defect-robust power domination study a fixed number of PMU failures.
        However, actual sensor failures occur randomly. 
        Thus, we assign a probability of failure to the PMUs to create the \emph{fragile power domination process}.
        By studying a model in which PMUs have a random chance of failure, the primary question is no longer to find the placement of a minimum number of sensors. 
        Instead, the \emph{expected observability} of the network is studied.
        
        In particular, we define the \emph{expected value polynomial} for a graph $G$ and PMU placement $S$ as a function of the PMU failure probability.
        This polynomial is shown to have a direct connection to $k$-fault-tolerant and $k$-PMU-defect-robust power domination.
        We use this polynomial to determine the probability that the entire graph will be observed.
        Methods of comparing this polynomial for different PMU placements are presented, and the polynomial is calculated for families of graphs.
        
        In \cref{sec:preliminaries}, we give definitions, establish useful binomial properties, and formally define relevant power domination variations.
        We then define the fragile power domination process and the associated expected value polynomial, and in \cref{sec:propertiesEGSq} explore properties of this polynomial.
        In \cref{sec:probability}, we consider the probability of observing the entire graph.
        Graph families are studied in \cref{sec:families}.
        Finally, in \cref{sec:futurework} we give ideas for future work.
    \section{Preliminaries}\label{sec:preliminaries}
        To begin, necessary graph theoretic definitions and some useful properties of binomial coefficients are established.
        The relevant notions of power domination, failed power domination, fault-tolerant power domination, and PMU-defect-robust power domination are given.
        Finally, we introduce fragile power domination and the expected value polynomial.
        \subsection{Graph Theory}
            A simple, undirected graph $G$ is a set $V(G)$ of vertices and a set $E(G)$ of edges.
            Each edge is an unordered pair of distinct vertices $\{x,y\}$, usually written $xy$. 
            When $xy$ is an edge, $x$ and $y$ are said to be \emph{adjacent} or \emph{neighbors}.
            A graph $H$ is a \emph{subgraph} of $G$ if $V(H) \subseteq V(G)$ and $E(H) \subseteq E(G)$.
            A graph $H$ is an \emph{induced subgraph} of $G$ if $V(H)\subseteq V(G)$ and $E(H)=\{xy\in E(G): x,y\in V(H) \}$ and is denoted by $G[V(H)]$.
            
            A \emph{path} from $v_1$ to $v_{\ell+1}$ is a sequence of distinct vertices $v_1,v_2,\ldots,v_{\ell+1}$ such that $v_iv_{i+1}$ is an edge for all $i=1,2,\ldots,\ell$.
            A graph $G$ is \emph{connected} if, for all $u,v\in V(G)$, there exists a path from $u$ to $v$.
            For a graph $G$ that is not connected, the subgraphs $H_1,\ldots,H_k$ are the \emph{connected components} of $G$ when $H_1,\ldots, H_k$ are connected and have no edges between them.
            
            A \emph{cycle} from $v_1$ to $v_{\ell}$ is a sequence of distinct vertices $v_1,v_2,\ldots,v_{\ell}$ such that $v_iv_{i+1}$ is an edge for all $i=1,2,\ldots,\ell$ and $v_1v_\ell$ is an edge.
            The \emph{neighborhood} of $u\in V(G)$, denoted by $N(u)$, is the set containing all neighbors of $u$.
            The \emph{closed neighborhood} of $u$ is $N[u]=N(u)\cup\{u\}$.
            Given an initial set $S$ of vertices of a graph $G$, $S$ is a \emph{dominating set} of $G$ if $\bigcup_{v\in S}N[v]=V(G)$.
            A \emph{universal vertex} $v$ has $N[v]=V(G)$.
            The \emph{degree} of a vertex $u$ is $\deg(u)=|N(u)|$.
            A \emph{leaf} is a vertex $v$ with $\deg(v)=1$.
            
            A \emph{subdivision} of an edge $xy$ creates a new vertex $w$ and replaces the edge $xy$ with the edges $xw$ and $wy$.
            If the edge $xy$ is subdivided $a$ times, vertices $w_1,w_2,\ldots,w_a$ are added to the vertex set and the edges $xw_1, w_1w_2,\ldots, w_{a-1}w_a, w_ay$ replace the edge $xy$.
            
            We denote the \emph{path on $n$ vertices} by $P_n$, the \emph{cycle on $n$ vertices} by $C_n$, and the \emph{complete graph on $n$ vertices} by $K_n$.
            The \emph{wheel on $n$ vertices}, denoted by $W_n$, is constructed by adding a universal vertex to $C_{n-1}$.
            A \emph{complete multipartite graph}, denoted by $K_{r_1,r_2,\ldots,r_k}$, has its vertex set partitioned into disjoint sets $V_1,V_2,\ldots,V_k$ with $|V_i|=r_i$, and edge set $\{xy : x\in V_i, y\in V_j, i\neq j\}$.
            The case of $k=2$ is the \emph{complete bipartite graph}.
            A special case of the complete bipartite graph is the \emph{star on $n$ vertices}, $S_n=K_{1,n-1}$.
        \subsection{Useful facts and notation}
            Some binomial identities will be used extensively and warrant being called out explicitly.
            The first two facts follow immediately from the binomial theorem, a proof of which can be found in Section 3.2 of \cite{CB_stats}.  
            \begin{fact}\label{fact1}
            Let $q\in[0,1]$ and $n\in\mathbb{N}$. Then
            \[\sum_{k=0}^n\binom{n}{k}q^{n-k}(1-q)^k=1.\]
            \end{fact}
            \begin{fact}\label{fact3}
            Let $n\in\mathbb{N}$. Then
            \[\sum_{k=0}^n (-1)^k\binom{n}{k}=0.\]
            \end{fact}
            We refer the reader to Section 2.2 of \cite{CB_stats} for proof of the following fact.
            \begin{fact}\label{fact2}
            Let $q\in[0,1]$ and $n\in\mathbb{N}$. Then
            \[\sum_{k=0}^n k\binom{n}{k}q^{n-k}(1-q)^k=n(1-q).\]
            \end{fact}
            
            Let $A$ and $B$ be a statistical events, that is, subsets of a finite sample space.
            We write $\operatorname{Prob} \big[ A \big]$ to denote the probability that $A$ occurs, and $\operatorname{Prob} \big[ A:B \big]$ to denote the probability of $A$ given that $B$ occurs.
            If $X$ is a random variable with possible outcomes $x_1,x_2,\ldots,x_n$, then $\{X=x_i\}$ is the event that $X$ takes the value $x_i$ and occurs with some probability $\operatorname{Prob} \big[ X=x_i \big]$.
            The expected value of the random variable is $\operatorname{\mathbb{E}} \left[ X \right]=\sum_{i=1}^n x_i\operatorname{Prob} \big[ X=x_i \big]$. 
            
            If $A_1$ and $A_2$ are events, $A_1\vee A_2$ denotes the event that at least one of $A_1$ or $A_2$ occurs.
            Additionally, $A_1\wedge A_2$ denotes the event that both $A_1$ and $A_2$ occur.
            When $A_1,\ldots,A_k$ are all events, the notations
            \[\bigvee_{i=1}^k A_i=A_1\vee\cdots\vee A_k
            \quad\text{and}\quad
            \bigwedge_{i=1}^k A_i=A_1\wedge\cdots\wedge A_k\]
            are used.
        \subsection{Power Domination}
            Formally, the \emph{power domination process} on a graph $G$ with initial set $S\subseteq V(G)$ proceeds recursively:
            \begin{enumerate}
            \item[1.] (\emph{Domination}) $B = \displaystyle \bigcup_{v\in S} N[v]$
            \item[2.] (\emph{Zero Forcing}) While there exists $v\in B$ such that exactly one neighbor, say $u$, of $v$ is \emph{not} in $B$, add $u$ to $B$. 
            \end{enumerate}
            During the process, a vertex in $B$ is \emph{observed} and a vertex not in $B$ is \emph{unobserved}.
            We denote the set of vertices observed at the termination of the power domination process by $\operatorname{Obs}\left(G;S\right)$.
            If $v$ causes $u$ to join $B$, then \emph{$v$ observes $u$}.
            A \emph{power dominating set} of a graph $G$ is an initial set $S$ such that $\operatorname{Obs}\left(G;S\right)=V(G)$.
            The \emph{power domination number} of a graph $G$ is the minimum cardinality of a power dominating set of $G$ and is denoted by $\gamma_P\left(G\right)$. 
            
            \emph{Failed power domination} is a related, reversed, notion to power domination introduced by Glasser et. al. in \cite{gjlr20} and will be very useful throughout this work. 
            \begin{defn}
            The \emph{failed power domination number} of a graph $G$ is the cardinality of a largest set $F \subseteq V(G)$ such that $\operatorname{Obs}\left(G;F\right) \neq V(G)$.
            This maximum cardinality is denoted by $\overline{\gamma}_P\left(G\right)$.
            \end{defn}
            
            \emph{Fault-tolerant power domination} was introduced in \cite{pcw10} in order to account for PMU failure when determining minimum power dominating sets. 
            \begin{defn}
            For a graph $G$ and an integer $k$ with $0\leq k < |V(G)|$, a set $S\subseteq V(G)$ is called a \emph{$k$-fault-tolerant power dominating set of $G$} if $S\setminus F$ is still a power dominating set of $G$ for any subset $F\subset V(G)$ with $|F|\leq k$.  The \emph{$k$-fault-tolerant power domination number} is the minimum cardinality of a $k$-fault-tolerant power dominating set of $G$. 
            \end{defn}
            
            \emph{PMU-defect-robust power domination} was introduced in \cite{bcf23} to extend vertex-fault tolerant power domination into multiset PMU placements.
            \begin{defn}\label{def:gpk}
                Let $G$ be a graph, $k \geq 0$ be an integer, and $S$ be a set or multiset whose elements are in $V(G)$.
                The set or multiset $S$ is a \emph{$k$-PMU-defect-robust power dominating set} ($k$-rPDS) if for any submultiset $F$ with $|F|=k$, $S \setminus F$ contains a power dominating set of vertices.
                The minimum cardinality of a $k$-rPDS is the $k$-PMU-defect-robust power domination number.
            \end{defn}
        \subsection{Fragile Power Domination}
            The \emph{fragile power domination process} is a variation of the power domination process with the addition that before the domination step, each PMU fails independently with probability $q$.
            We utilize the same notation and terminology as the power domination process, given in the previous section, with suitable modifications.
            \newpage
            \begin{defn}
                The \emph{fragile power domination process} on a graph $G$ with initial set or multiset of vertices $S$ is as follows:
                \begin{enumerate}
                    \item[0.]
                        \emph{(Sensor Failure)} Let $S^* = \varnothing$. For each $v\in S$, add $v$ to $S^*$ with probability $1-q$
                    \item[1.]
                        \emph{(Domination)} $B = \displaystyle \bigcup_{v\in S^*} N[v]$
                    \item[2.]
                        \emph{(Zero Forcing)} While there exists $v\in B$ such that exactly one neighbor, say $u$, of $v$ is \emph{not} in $B$, add $u$ to $B$. 
                \end{enumerate}
            \end{defn}
            
            Given initial PMU placement $S$, the set or multiset $S^*$ is the random collection of vertices which did not fail with probability $1-q$.
            We denote the set of vertices observed at the termination of the fragile power domination process by $\operatorname{Obs}\left(G;S,q\right)$.
            Observe that $|\operatorname{Obs}\left(G;S,q\right)|$ is a random variable and $\operatorname{Obs}\left(G;S\right)=\operatorname{Obs}\left(G;S,0\right)$.
            Moreover, for a fixed set or multiset of vertices $S$, the expected value of $|\operatorname{Obs}\left(G;S,q\right)|$ is a polynomial in $q$.
            \begin{defn}
                For a given graph $G$, set or multiset of vertices $S$, and probability of PMU failure $q$, we define the \emph{expected value polynomial} to be $$\mathcal{E}\left(G;S,q\right)=\operatorname{\mathbb{E}} \left[ |\operatorname{Obs}\left(G;S,q\right)| \right],$$ 
                the expected value of the random variable $|\operatorname{Obs}\left(G;S,q\right)|$.
            \end{defn}
            This polynomial will serve as a central tool for investigating fragile power domination.
            \begin{obs}\label{obs:calculateexpolbinomial}
                $\mathcal{E}\left(G;S,q\right)$, as a polynomial in $q$ with degree at most $|S|$, can be calculated via:
                \begin{align*}
                    \mathcal{E}\left(G;S,q\right) = \sum_{W \subseteq S}\big|\operatorname{Obs}\left(G;W\right)\big|q^{|S \setminus W|}(1-q)^{|W|}
                \end{align*}
            \end{obs}
            In the next example we explore $\mathcal{E}\left(G;S,q\right)$ for all $2$-multisets of a specific graph.
            \begin{ex}\label{Example:jspectrumexample}
                Let $G$ be the graph $G$ in \cref{fig:spectrumplot}.
                Note that $\gamma_P\left(G\right)=2$.
                The 28 unique placements of 2 PMUs, corresponding to the 28 multisets of 2 vertices, result in 10 distinct expected value functions, which are plotted in
                \cref{fig:spectrumplot}:
                \begin{itemize}
                    \item 
                        $\mathcal{E}\left(G;S,q\right) = 4(1-q)q + 2(1-q)^2$ when $S=\{2,2\}, \{3,3\}, \{6,6\}$ or $ \{7,7\}$
                    \item 
                        $\mathcal{E}\left(G;S,q\right) = 4(1-q)q + 3(1-q)^2$ when $S=\{2,3\}$ or $ \{6,7\}$
                    \item 
                        $\mathcal{E}\left(G;S,q\right) = 4(1-q)q + 4(1-q)^2$ when $S=\{2,6\}, \{2,7\}, \{3,6\},$ or $ \{3,7\}$
                    \item 
                        $\mathcal{E}\left(G;S,q\right) = 6(1-q)q + 3(1-q)^2$ when $S=\{1,1\}$
                    \item 
                        $\mathcal{E}\left(G;S,q\right) = 5(1-q)q + 5(1-q)^2$ when $S=\{1,2\}, \{1,3\}, \{1,6\},$ or $ \{1,7\}$
                    \item 
                        $\mathcal{E}\left(G;S,q\right) = 7(1-q)q + 5(1-q)^2$ when $S=\{2,4\}, \{3,4\}, \{5,6\},$ or $ \{5,7\}$
                    \item 
                        $\mathcal{E}\left(G;S,q\right) = 8(1-q)q + 5(1-q)^2$ when $S=\{1,4\}$ or $ \{1,5\}$
                    \item 
                        $\mathcal{E}\left(G;S,q\right) = 10(1-q)q + 5(1-q)^2$ when $S=\{4,4\}$ or $ \{5,5\}$
                    \item 
                        $\mathcal{E}\left(G;S,q\right) = 7(1-q)q + 7(1-q)^2$ when $S=\{2,5\}, \{3,5\}, \{4,6\},$ or $ \{4,7\}$
                    \item 
                        $\mathcal{E}\left(G;S,q\right) = 10(1-q)q + 7(1-q)^2$ when $S=\{4,5\}$.
                \end{itemize}
            \end{ex}
            \begin{figure}[ht]
                \centering
                \includegraphics{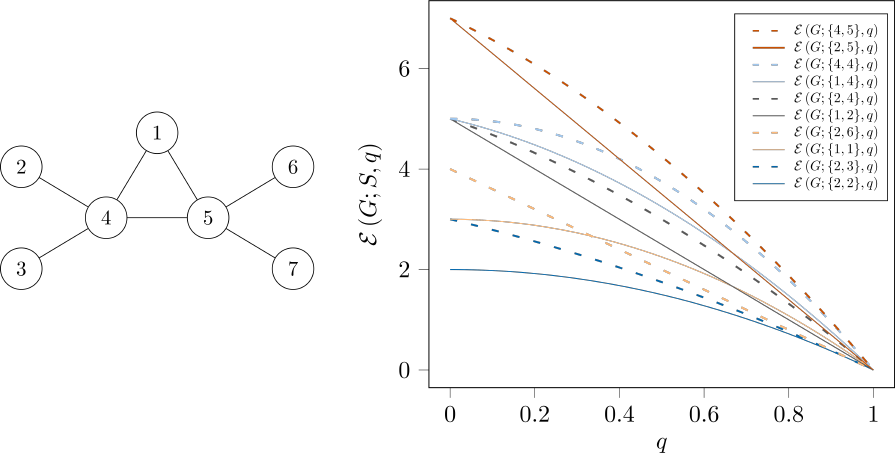}
                \caption{A graph $G$ and all distinct $\mathcal{E}\left(G;S,q\right)$ for 2-multisets $S$ containing vertices of $G$.}
                \label{fig:spectrumplot}
            \end{figure}
            Overall, the power dominating set $\{4,5\}$ is expected to observe more vertices than any other placement of two PMUs.
            The only other power dominating sets of size two are $\{2,5\}$, $\{3,5\}$, $\{4,6\}$, and $ \{4,7\}$, all of which have the same expected value polynomial.
            Notice that for $q\in\left(\nicefrac{2}{5},1\right)$ the placement of $\{4,4\}$, a failed power dominating set, is expected to observe more vertices than the placement of $\{2,5\}$, a power dominating set.
            We will explore the intersection of the expected value polynomials for power dominating sets and failed power dominating sets in \cref{sec:fishie}.
            
            \cref{Example:jspectrumexample} demonstrates $\mathcal{E}\left(G;S,q\right)$ for different PMU placements $S$ and a range of $q$ values.
            A natural baseline comparison for these polynomials is a placement of PMUs for which the observed vertices are roughly evenly distributed.
            That is, let $G$ be a graph on $n$ vertices and let $S$ be a power dominating set of $G$ such that, on average, $\frac{n}{\gamma_P\left(G\right)}$ vertices can be observed for each individual PMU in $S$.
            Then using \cref{fact2} we find
            \begin{align*}
                \mathcal{E}\left(G;S,q\right) &= \sum_{W\subseteq S} \left( |W| \, \frac{n}{\gamma_P\left(G\right)} \right) q^{\gamma_P\left(G\right)-|W|} (1-q)^{|W|} \\
                &= \frac{n}{\gamma_P\left(G\right)} \sum_{|W|=0}^{\gamma_P\left(G\right)} |W| \binom{\gamma_P\left(G\right)}{|W|} q^{\gamma_P\left(G\right)-|W|} (1-q)^{|W|} \\
                &= \frac{n}{\gamma_P\left(G\right)} \, \gamma_P\left(G\right) (1-q) \\
                &= n(1-q).
            \end{align*}
            In this sense, $\mathcal{E}\left(G;S,q\right)$ being linear in $q$ corresponds to a roughly average PMU covering of $G$.
            This notion is formalized in \cref{prop:LinearExp}.
    \section{Properties of the expected value polynomial}\label{sec:propertiesEGSq}
        As in power domination, fragile power domination on a disconnected graph $G$ can be observed as the sub-problems on the connected components of $G$.
        Particularly, given a graph $G$ with connected components $H_1,\ldots,H_k$ and initial set or multiset of vertices $S$, $\mathcal{E}\left(G;S,q\right)=\sum_{i=1}^{k}\mathcal{E}\left(H_i;V(H_i)\cap S,q\right)$.
        Thus, from now on, all graphs are assumed to be connected.
        \subsection{Linearity}
            With $\mathcal{E}\left(G;S,q\right)$ having degree at most $|S|$, it is trivial to show that when $|S|=1$ then $\mathcal{E}\left(G;S,q\right)$ is linear.
            \begin{obs}
                For a graph $G$ on $n$ vertices with $\gamma_P\left(G\right)\geq1$, a power dominating set or multiset $S$, and probability of PMU failure $q$, $\mathcal{E}\left(G;S,q\right)\geq n(1-q)^{|S|}$. 
                Equality holds when $|S|=1$, resulting in a linear $\mathcal{E}\left(G;S,q\right)$.
            \end{obs}
            To see this, notice 
            \begin{align*}
                \mathcal{E}\left(G;S,q\right) &= \sum_{W\subseteq S}\big|\operatorname{Obs}\left(G;W\right)\big|q^{|S\setminus W|}(1-q)^{|W|}\\
                &= n(1-q)^{|S|}+\sum_{W \subsetneq S}q^{|S\setminus W|}(1-q)^{|W|}.
            \end{align*}
            The following yields another condition for when $\mathcal{E}\left(G;S,q\right)$ is linear.
            \begin{prop}\label{prop:LinearExp}
                Let $G$ be a graph, $S\subseteq V(G)$ be a set, and $q$ be a probability of PMU failure.
                If for every $X\subseteq S$,
                \[|\operatorname{Obs}\left(G;X\right)|=\sum_{v\in X} |\operatorname{Obs}\left(G;\{v\}\right)|,\]
                then $\mathcal{E}\left(G;S,q\right)=|\operatorname{Obs}\left(G;S\right)|(1-q).$
            \end{prop}
            \begin{proof}
                We show this by induction. Observe that $|S|=1$ is immediate. 
                
                Now suppose the claim holds for vertex sets of size $k-1$ and let $S\subseteq V(G)$ be of size $k$ such that \[\big|\operatorname{Obs}\left(G;X\right)\big|=\sum_{v\in X} \big|\operatorname{Obs}\left(G;\{v\}\right)\big|\]
                for every $X\subseteq S$.
                Fix $v\in S$ and let $Y=S\setminus \{v\}$.
                Then $\mathcal{E}\left(G;S,q\right)$ is equal to
                \begin{align*}
                    q\sum_{W\subseteq Y}\big|\operatorname{Obs}\left(G;W\right)\big|\,q^{|Y\setminus W|}(1-q)^{|W|}+(1-q)\sum_{W\subseteq Y}\big|\operatorname{Obs}\left(G;W\cup\{v\}\right)\big|\,q^{|Y\setminus W|}(1-q)^{|W|}
                \end{align*}
                and using the assumption on $S$ we can split $|\operatorname{Obs}\left(G;W\cup\{v\}\right)|=|\operatorname{Obs}\left(G;W\right)|+|\operatorname{Obs}\left(G;\{v\}\right)|$, resulting in
                \[\sum_{W\subseteq Y}\big|\operatorname{Obs}\left(G;W\right)\big|\,q^{|Y\setminus W|}(1-q)^{|W|}+(1-q)\sum_{W\subseteq Y}\big|\operatorname{Obs}\left(G;\{v\}\right)\big|\,q^{|Y\setminus W|}(1-q)^{|W|}\]
                by combining the $q$ and $1-q$ sums. This is now equal to
                \[\mathcal{E}\left(G;Y,q\right)+(1-q)\big|\operatorname{Obs}\left(G;\{v\}\right)\big|\sum_{i=0}^{|Y|}\binom{|Y|}{i}q^{|Y|-i}(1-q)^i.\]
                Then by the inductive hypothesis $\mathcal{E}\left(G;Y,q\right)=(1-q)|\operatorname{Obs}\left(G;S\setminus\{v\}\right)|$ and so applying \cref{fact1} yields
                \[(1-q)\big|\operatorname{Obs}\left(G;S\setminus\{v\}\right)\big|+(1-q)\big|\operatorname{Obs}\left(G;\{v\}\right)|
                =\big|\operatorname{Obs}\left(G;S\right)\big|(1-q).\qedhere\]
            \end{proof}
            
            We now construct a family of graphs for which the unique minimum power dominating set satisfies the assumption of \cref{prop:LinearExp}.
            Start with the complete multipartite graph $K_{r_1,r_2,\ldots,r_k}$ with $k\geq 2$ and $r_i \geq 2$ for all $i$, and vertex partitions $V_1,V_2,\ldots,V_k$.
            For each partition $V_i$, add a vertex $a_i$ adjacent to all vertices in $V_i$. 
            Then for each $a_i$ add two adjacent leaves.
            If desired, subdivide any edge incident to $a_i$ any number of times.
            The set $\{a_1,a_2,\ldots,a_k\}$ is the unique minimum power dominating set and satisfies the hypothesis of \cref{prop:LinearExp}.
            The family of graphs constructed from $K_{2,2}$ is shown in \cref{fig:LinearConstruct}.
            \begin{figure}[H]
                \centering
                \includegraphics{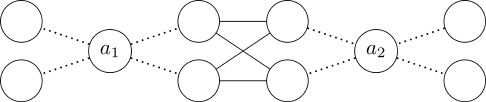}
                \caption{The family of graphs built from $K_{2,2}$ where the minimum power dominating set $\{a_1,a_2\}$ satisfies \cref{prop:LinearExp}. The dotted lines indicate edges that may be subdivided.}
                \label{fig:LinearConstruct}
            \end{figure}
            
            A consequence of the proof of \cref{prop:LinearExp} is a condition for when the degree of the expected value polynomial does not change after a vertex is added to a given PMU placement.
            \begin{cor}
                Let $G$ be a graph and let $q$ be a probability of PMU failure.
                Let $S$ be a set or multiset of vertices and $v\in V(G)\setminus S$.
                If for all $W\subseteq S$, $|\operatorname{Obs}\left(G;W\cup\{v\}\right)|=|\operatorname{Obs}\left(G;W\right)|+|\operatorname{Obs}\left(G;\{v\}\right)|$, then $\mathcal{E}\left(G;S,q\right)$ and $\mathcal{E}\left(G;S\cup\{v\},q\right)$ are the same degree. 
            \end{cor}
            A natural question to ask is whether, for all $q$, a power dominating set is better or worse than the linear expected value given by \cref{prop:LinearExp}. 
            When $q$ is large, it is important that each PMU individually observes as many vertices as possible.
            One way to formalize this idea is through the following definition.
            \begin{defn}\label{def:localcover}
                Let $G$ be a graph and let $S$ be a set or multiset of the vertices of $G$.
                Then we call $S$ a \emph{local cover of $G$} if for every $x \in V(G)$, there exists a vertex $v \in S$ such that $x \in \operatorname{Obs}\left(G;\{v\}\right)$.
            \end{defn}
            The notion of a local cover gives an alternative to \cref{prop:LinearExp}.
            \begin{prop}\label{prop:localcover}
                Let $G$ be a graph on $n$ vertices and let $q$ be a probability of PMU failure.
                If $S$ is a local cover of $G\big[\operatorname{Obs}\left(G;S\right)\big]$, then $\mathcal{E}\left(G;S,q\right)\geq |\operatorname{Obs}\left(G;S\right)|(1-q)$ for all $q$.
                In particular, if $S$ is a power dominating set then $\mathcal{E}\left(G;S,q\right)\geq n(1-q)$ for all $q$.
            \end{prop}
            \begin{proof}
                Assume that $\operatorname{Obs}\left(G;S\right)=\displaystyle\bigcup_{v\in S}^{} \operatorname{Obs}\left(G;\{v\}\right)$.
                We calculate $\mathcal{E}\left(G;S,q\right)$ as the following:
             \begin{align*}
                 \sum_{W\subseteq S}\big|\operatorname{Obs}\left(G;W\right)\big|\,q^{|S\setminus W|}(1-q)^{|W|}
                 =\sum_{j=1}^{|S|} q^{|S|-j}(1-q)^{j} \sum_{\substack{W\subseteq S:\\|W|=j}}|\operatorname{Obs}\left(G;W\right)|.
             \end{align*}
                Now let $x \in \operatorname{Obs}\left(G;S\right)$ and choose $v \in S$ such that $x \in \operatorname{Obs}\left(G;\{v\}\right)$.
                The number of subsets $W \subseteq S$ of size $j$ containing the vertex $v$ is $\binom{|S|-1}{j-1}$.
                The sum of all the observed vertices for any subset $W$ of size $j$ counts the vertex $x$ at least $\binom{|S|-1}{j-1}$ times and therefore for any $q$,
                \begin{align*}
                    \mathcal{E}\left(G;S,q\right)&\geq \sum_{j=1}^{|S|} |\operatorname{Obs}\left(G;S\right)| \binom{|S|-1}{j-1}q^{|S|-j}(1-q)^{j}\\
                    &= |\operatorname{Obs}\left(G;S\right)|(1-q) \sum_{j=1}^{|S|} \binom{|S|-1}{j-1}q^{(|S|-1)-(j-1)}(1-q)^{j-1}\\
                    &= |\operatorname{Obs}\left(G;S\right)|(1-q) \sum_{k=0}^{|S|-1}\binom{|S|-1}{k}q^{(|S|-1)-k}(1-q)^{k}
                \end{align*}
                which is equal to $|\operatorname{Obs}\left(G;S\right)|(1-q)$ by \cref{fact1}. 
            \end{proof}
            The following is an immediate corollary for dominating sets.
            \begin{cor}
                For a graph $G$ on $n$ vertices, dominating set $D$ of $G$, and probability of PMU failure $q$, $\mathcal{E}\left(G;D,q\right) \geq n(1-q)$ for all $q$. 
            \end{cor}
            Notice that the converse of \cref{prop:localcover} is not true.
            This can be seen by the graph $G$ in \cref{fig:balloondog}.
            For this graph, 
            \[ \mathcal{E}\left(G;\{a,b,c\},q\right) = 16(1-q)^3+40q(1-q)^2+23q^2(1-q). \]
            For all $q\in[0,1]$, $\mathcal{E}\left(G;\{a,b,c\},q\right) \geq 16(1-q)$, but $\{a,b,c\}$ is not a local cover of $G$ as $x$ is not observed by any individual vertex in $\{a,b,c\}$.
            Moreover, being a $k$-rPDS for $k \geq 1$ does not guarantee being a local cover.
            For example, the set of leaves of $G$ is a $1$-rPDS because a set consisting of any six of the seven leaves forms a power dominating set, but vertex $z$ is not in $\operatorname{Obs}\left(G;\{v\}\right)$ for any leaf $v$ of $G$.
            In fact, the only minimum local covers of $G$ are $\{a, b, c, x\}$ and $\{a, b, c, y\}$.
            \begin{figure}[ht]
                \centering
                \includegraphics{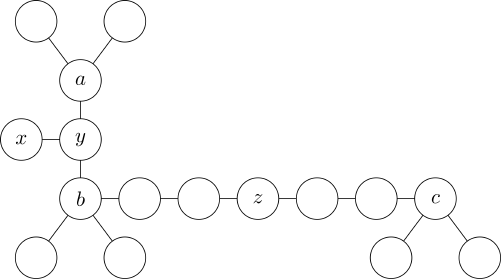}
                \caption{A graph $G$ with a minimum power dominating set $S=\{a, b, c\}$ for which $\mathcal{E}\left(G;S,q\right) \geq n(1-q)$ for $q\in[0,1]$, but $S$ is not a local cover.}
                \label{fig:balloondog}
            \end{figure}
        \subsection{Coefficients}
            In this section, we will explore the relationship between the structure of the expected value polynomial and and properties of its fixed PMU placements.
            Specifically, we show a correspondence between the coefficients of $\mathcal{E}\left(G;S,q\right)$ and whether the PMU placement $S$ is PMU-defect-robust or fault-tolerant.
            We also show a connection to the failed power domination number.
            We begin with the following result which calculates the coefficients of $\mathcal{E}\left(G;S,q\right)$ in standard form.
            \begin{lem}\label{lem:ExpCoefficients}
                Let $G$ be a graph and $q$ a probability of PMU failure.
                Let $S$ be a set or multiset of vertices and $r\in\{0,\ldots,|S|\}$.
                Then the coefficient of $q^{|S|-r}$ in $\mathcal{E}\left(G;S,q\right)$ is
                \[\sum_{\substack{W\subseteq S:\\|W|\geq r}}(-1)^{|W|-r}\binom{|W|}{r}|\operatorname{Obs}\left(G;W\right)|.\]
            \end{lem}
            \begin{proof}
                Let $G$ be a graph.
                Notice
                \begin{align*}
                    \mathcal{E}\left(G;S,q\right)&=\sum_{W\subseteq S}|\operatorname{Obs}\left(G;W\right)|q^{|S\setminus W|}(1-q)^{|W|}\\
                    &=\sum_{W\subseteq S}|\operatorname{Obs}\left(G;W\right)|q^{|S\setminus W|}\sum_{i=1}^{|W|}\binom{|W|}{i}(-q)^i\\
                    &=\sum_{W\subseteq S}\sum_{i=0}^{|W|}(-1)^i\binom{|W|}{i}|\operatorname{Obs}\left(G;W\right)|q^{|S\setminus W|+i}
                \end{align*}
                and so we contribute to the coefficient of $q^{|S|-r}$ whenever $|W|\geq r$ and $i=|W|-r$.
                Hence the coefficient of $q^{|S|-r}$ in $\mathcal{E}\left(G;S,q\right)$ is
                \[\sum_{\substack{W\subseteq S:\\|W|\geq r}}(-1)^{|W|-r}\binom{|W|}{r}|\operatorname{Obs}\left(G;W\right)|.\qedhere\]
            \end{proof}
            \cref{lem:ExpCoefficients} gives a complete characterization of $k$-PMU-defect-robust power dominating sets or multisets.
            \begin{thm}\label{thm:robustbyexpoly}
                Let $G$ be a graph on $n$ vertices, $q$ be a probability of PMU failure, and the set or multiset of vertices $B$ be an initial placement of PMUs with $|B|=b$.
                Then $B$ is a $k$-PMU-defect-robust power dominating set of $G$ if and only if $\mathcal{E}\left(G;B,q\right)$ has the form
                \[\mathcal{E}\left(G;B,q\right)=n-q^{k+1}h(q)\]
                for some polynomial $h(q)$.
            \end{thm}
            \begin{proof}
                First assume $B$ is a $k$-rPDS.
                This means any subset $A\subseteq B$ of cardinality $|A|\geq b-k$ is a power dominating set, that is $\operatorname{Obs}\left(G;A\right)=V(G)$.
                Let $C_w$ denote the coefficient of $q^w$ in $\mathcal{E}\left(G;B,q\right)$ as given by \cref{lem:ExpCoefficients}.
                Then notice for any $1\leq w\leq k$,
                \begin{align*}
                    C_w=\sum_{\substack{A\subseteq B:\\|A|\geq b-w}}(-1)^{|A|-(b-w)}\binom{|A|}{b-w}|\operatorname{Obs}\left(G;A\right)|
                    &=n\sum_{\substack{A\subseteq B:\\|A|\geq b-w}}(-1)^{|A|-(b-w)}\binom{|A|}{b-w}\\
                    &=n\sum_{j=b-w}^b (-1)^{j-(b-w)}\binom{b}{j}\binom{j}{b-w}\\
                    &=n\sum_{i=0}^w (-1)^{i}\binom{b}{i+b-w}\binom{i+b-w}{b-w}\\
                    &=n\binom{b}{w}\sum_{i=0}^w (-1)^{i}\binom{w}{i}\\
                    &=0
                \end{align*}
                since $\operatorname{Obs}\left(G;A\right)=V(G)$ for all $A\subseteq B$ with $|A|\geq b-w\geq b-k$ and
                \begin{align*}
                    \binom{b}{i+b-w}\binom{i+b-w}{b-w}
                    =\frac{b!}{(w-i)!\,(b-w)!\,i!}
                    =\binom{w}{i}\binom{b}{w}.
                \end{align*}
                Thus
                \begin{align*}
                    \mathcal{E}\left(G;B,q\right)&=C_0+C_{k+1}\,q^{k+1}+C_{k+2}\,q^{k+2}+\cdots+C_b\, q^b\\
                    &=n-q^{k+1}\big(-C_{k+1}-C_{k+2}\, q-\cdots-C_b\, q^{b-(k+1)}\big).
                \end{align*}
                Now let $B\subseteq V(G)$ with $|B|=b$ and suppose $\mathcal{E}\left(G;B,q\right)=n-q^{k+1} h(q)$ with $C_i$ denoting the coefficient of $q^i$.
                We prove $B$ is a $k$-rPDS by a recursive argument.
                For the base case, let $k=1$ with polynomial $\mathcal{E}\left(G;B,q\right)=n-q^2 h(q)$.
                Notice $C_0=n$ implies $|\operatorname{Obs}\left(G;B\right)|=n$ and so $B$ is a power dominating set.
                Moreover, by \cref{lem:ExpCoefficients} 
                \[0=C_1=\sum_{\substack{A\subseteq B:\\|A|\geq b-1}}(-1)^{|A|-(b-1)}\binom{|A|}{b-1}|\operatorname{Obs}\left(G;A\right)|=-bn+\sum_{\substack{A\subseteq B:\\|A|=b-1}}|\operatorname{Obs}\left(G;A\right)|\]
                which implies 
                \[\sum_{\substack{A\subseteq B:\\|A|=b-1}}|\operatorname{Obs}\left(G;A\right)|=bn.\]
                Since there are $b$ summands and $|\operatorname{Obs}\left(G;S\right)|\leq n$ for all $S\subseteq V(G)$, this can only happen if $|\operatorname{Obs}\left(G;A\right)|=n$ for all $A\subseteq B$ of cardinality $b-1$.
                In other words, $B$ is a $1$-rPDS. 
                
                Now let $\mathcal{E}\left(G;B,q\right)=n-q^{k+1} h(q)$ with $1<k\leq b$ and assume the claim holds for $t\leq k-1$.
                Then since $\mathcal{E}\left(G;B,q\right)=n-q^k g(q)$ for $g(q)=qh(q)$, $B$ is a $t$-rPDS for $0\leq t\leq k-1$.
                That is, $|\operatorname{Obs}\left(G;A\right)|=n$ for all $A\subseteq B$ such that $|A|\geq b-(k-1)$.
                Then notice that
                \begin{align*}
                    0=C_k&=\sum_{\substack{A\subseteq B:\\|A|\geq b-k}} (-1)^{|A|-(b-k)}\binom{|A|}{b-k}|\operatorname{Obs}\left(G;A\right)|\\
                    &=n\sum_{\substack{A\subseteq B:\\|A|\geq b-(k-1)}}(-1)^{|A|-(b-k)}\binom{|A|}{b-k}
                    + \sum_{\substack{A\subseteq B:\\|A|= b-k}} (-1)^{|A|-(b-k)}\binom{|A|}{b-k}|\operatorname{Obs}\left(G;A\right)|\\
                    &=n\sum_{i=1}^k (-1)^i \binom{b}{k}\binom{k}{i}
                    +\sum_{\substack{A\subseteq B:\\|A|= b-k}}|\operatorname{Obs}\left(G;A\right)|
                \end{align*}
                where in the last equality we use the same transformation as in the proof of the forward direction.
                Now it follows from \cref{fact3} that
                \[n\binom{n}{k}\sum_{i=0}^k (-1)^i\binom{k}{i}=0=n\sum_{i=1}^k (-1)^i \binom{b}{k}\binom{k}{i}
                +\sum_{\substack{A\subseteq B:\\|A|= b-k}}|\operatorname{Obs}\left(G;A\right)|\]
                and thus
                \[n\binom{b}{k}=\sum_{\substack{A\subseteq B:\\|A|= b-k}} |\operatorname{Obs}\left(G;A\right)|.\]
                Finally, because there are $\binom{b}{b-k}=\binom{b}{k}$ choices for subsets $A\subseteq B$ of size $b-k$, we have that each $A$ must observe $n$ vertices.
                Hence $B$ is a $k$-rPDS.
            \end{proof}
            Notice that a set which is $k$-fault-tolerant is also $k$-PMU-defect-robust, and so \cref{thm:robustbyexpoly} gives us the following corollary.
            
            \begin{cor}\label{thm:faultolerantbyexpoly}
                Let $G$ be a graph on $n$ vertices, $q$ be a probability of PMU failure, and the set $S\subseteq V(G)$ be an initial placement of PMUs.
                Then $S$ is a $k$-fault-tolerant power dominating set of $G$ if and only if $\mathcal{E}\left(G;S,q\right)$ has the form
                \[\mathcal{E}\left(G;S,q\right)=n-q^{k+1}h(q)\]
                for some polynomial $h(q)$.
            \end{cor}
            The largest set that can be used in $k$-fault-tolerant power domination is the entire vertex set.
            This means that $V(G)$ determines the largest possible $k$ for which $G$ can be $k$-fault-tolerant.
            We can determine this $k$ using the expected value polynomial.
            \begin{prop}\label{thm:faultolerantkbound}
                Let $G$ be a graph on $n$ vertices and let $q$ be a probability of PMU failure.
                Writing
                $\mathcal{E}\left(G;V(G),q\right)=n-q^{k+1}h(q)$
                for some polynomial $h(q)$ containing a nonzero constant term, it follows that $k$ is the largest possible $k$ for which $G$ can be $k$-fault-tolerant.
            \end{prop}
            Notice that if the entire vertex set is at most $k$-PMU-defect-robust, then there must exist some failed power dominating set $F$ with $|F|=n-k-1$. Then by \cref{thm:robustbyexpoly}, we obtain a similar structural constraint on the expected value polynomial in terms of the failed power domination number.
            \begin{cor}
                For a graph $G$ on $n$ vertices such that $\mathcal{E}\left(G;V(G),q\right)$ has the form $n-q^{k+1}h(q)$ with $h(q)$ containing a nonzero constant term, we have $\overline{\gamma}_P\left(G\right)=n-k-1$.
            \end{cor}
        \subsection{Comparisons}\label{sec:fishie}
            Now we will use the expected value polynomial as a tool to compare PMU placements.
            We examine what happens when a PMU is added to an existing placement.
            Then we compare failed power dominating sets to power dominating sets of the same size.
            \begin{lem}\label{lem:Esubsetmonotone}
                Let $G$ be a graph, $S$ a set or multiset of vertices of $G$, and $q$ a probability of PMU failure.
                Given a vertex $v \in V(G)$, it holds that $\mathcal{E}\left(G;S,q\right) \leq \mathcal{E}\left(G;S\cup\{v\},q\right)$.
            \end{lem}
            \begin{proof}
                By definition, $\mathcal{E}\left(G;S\cup\{v\},q\right)$ equals
                \begin{align*}
                    \sum_{W\subseteq S}|\operatorname{Obs}\left(G;W\right)|q^{|S\setminus W|+1}(1-q)^{|W|} + 
                \sum_{W\subseteq S}|\operatorname{Obs}\left(G;W \cup \{v\}\right)|q^{|S\setminus W|}(1-q)^{|W|+1}.
                \end{align*}
                where the first sum corresponds to subsets or submultisets of $S$, and the second sum to those which add $v$.
                Note that the first sum reduces to $q\, \mathcal{E}\left(G;S,q\right)$ and the second sum reduces to $(1-q) \displaystyle\sum_{W\subseteq S}|\operatorname{Obs}\left(G;W \cup \{v\}\right)|*q^{|S\setminus W|}(1-q)^{|W|}$.
                The set $W \cup \{v\}$ observes all the vertices observed by the set $W$, so $|\operatorname{Obs}\left(G;W \cup \{v\}\right)| \geq |\operatorname{Obs}\left(G;W\right)|$.
                Thus, 
                \begin{align*}
                    \mathcal{E}\left(G;S\cup\{v\},q\right) &\geq q \mathcal{E}\left(G;S,q\right)+(1-q)\mathcal{E}\left(G;S,q\right)\\
                    &= \mathcal{E}\left(G;S,q\right). \qedhere
                \end{align*}
            \end{proof}
            Restricting $S$ to be a set results in the following upper bound for the expected value polynomial.
            \begin{cor}\label{cor:wholethingbestthing}
                Let $G$ be a graph, $S\subseteq V(G)$ a set, and $q$ a probability of PMU failure.
                Then $\mathcal{E}\left(G;S,q\right) \leq \mathcal{E}\left(G;V(G),q\right)$.
            \end{cor}
            We now compare a fixed set or multiset of vertices $S$ to those of size $|S|+1$, including ones which do not contain $S$.
            Consider the graph $G$ in \cref{Figure:11node}.
            The best placement of $1$ PMU is $\{6\}$ for all values of $q\in[0,1]$.
            Considering sets and multisets $S$ with $|S|=2$ and $\{6\} \subset S$ we calculate the following expected value polynomials:
            \begin{itemize}
                \item
                    $\mathcal{E}\left(G;S,q\right)=10q(1-q) + 7(1-q)^2$ when $S=\{2, 6\}, \{3, 6\}, \{6, 9\},$ or $\{6, 10\}$
                \item
                    $\mathcal{E}\left(G;S,q\right)=9q(1-q) + 9(1-q)^2$ when $S=\{1, 6\}, \{4, 6\}, \{6, 8\},$ or $\{6, 11\}$
                \item
                    $\mathcal{E}\left(G;S,q\right)=14q(1 - q) + 7(1 - q)^2$ when $S=\{6,6\}$
                \item
                    $\mathcal{E}\left(G;S,q\right)=13q(1-q) + 9(1-q)^2$ when $S=\{5, 6\}$ or $\{6, 7\}$.
            \end{itemize}
            The placement $\{5, 7\}$, corresponding to $\mathcal{E}\left(G;S,q\right)=12q(1-q) + 11(1-q)^2$, is expected to observe more vertices of $G$ for $q \in [0,\nicefrac{2}{3})$ than any other 2-set or multiset, including those containing vertex $6$.
            Therefore, even if a set or multiset of vertices $S$ is the best PMU placement using $|S|$ PMUs for all values of $q$, it is not the case that the best PMU placement using $|S|+1$ PMUs is of the form $S\cup\{v\}$ for some $v\in V(G)$.
            \begin{figure}[ht]
                \centering
                \includegraphics{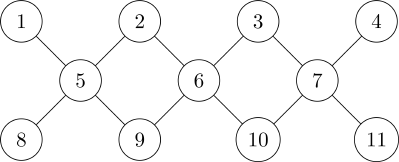}
                \caption{Small graph $G$ on 11 vertices.}
                \label{Figure:11node}
            \end{figure}
            
            It is natural to believe that once $|S| \geq \gamma_P\left(G\right)$, then $\mathcal{E}\left(G;S,q\right)$ will always be bounded above by a power dominating set.
            Unfortunately, the following describes how this is not always the case.
            \begin{thm}\label{thm:fishie}
                For any $q_*\in(0,1)\cap\mathbb{Q}$, there exists a graph $G$, failed power dominating set $F$, and power dominating set $S$ with $|F|=|S|$ such that $\mathcal{E}\left(G;F,q\right)\leq\mathcal{E}\left(G;S,q\right)$ for all $0\leq q\leq q_*$ and $\mathcal{E}\left(G;S,q\right)\leq\mathcal{E}\left(G;F,q\right)$ for all $q_*\leq q\leq 1$.
            \end{thm}
            \begin{figure}[ht]
                \centering
                \includegraphics{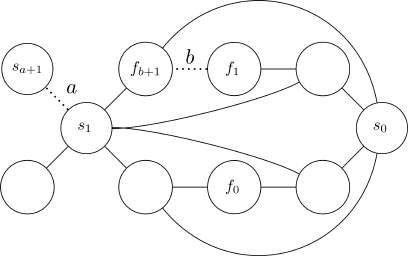}
                \caption{The family of graphs $G_{a,b}$ where the edge $s_1s_{a+1}$ was subdivided $a$ times and the edge $f_1f_{b+1}$ was subdivided $b$ times. Note that $\{s_0,s_1\}$ is a power dominating set and $\{f_0,f_1\}$ is a failed power dominating set.}
                \label{fig:zims_cousin}
            \end{figure}
            \begin{proof}
                Consider the graph $G_{a,b}$ as in \cref{fig:zims_cousin} where the edge $s_1s_{a+1}$ is subdivided $a\geq 1$ times and the edge $f_1f_{b+1}$ is subdivided $b\geq1$ times.
                It can be verified that $S=\{s_0,s_1\}$ is a power dominating set that observes all $10+a+b$ vertices and $F=\{f_0,f_1\}$ is a failed power dominating set that observes only $6+b$ vertices.
                We show that $a$ and $b$ can be chosen so that the unique intersection of $\mathcal{E}\left(G_{a,b};S,q\right)$ and $\mathcal{E}\left(G_{a,b};F,q\right)$ is any $q_*\in(0,1)\cap\mathbb{Q}$.
                Observe that
                \begin{align*}
                    \mathcal{E}\left(G_{a,b};S,q\right) &= (10 + a + b)(1-q)^2 + (12 + a)(1-q)q
                \end{align*}
                and
                \begin{align*}
                    \mathcal{E}\left(G_{a,b};F,q\right) &= (6 + b)(1-q)^2 + (6 + b)(1-q)q.
                \end{align*}
                Let $g(q)=\mathcal{E}\left(G_{a,b};S,q\right) - \mathcal{E}\left(G_{a,b};F,q\right)$ and notice $g$ can be simplified to $$g(q)=(q-1)\big((b-2)q-(a+4)\big).$$
                In this form, it can be seen that the zeros of $g$ are $1$ and $\frac{a+4}{b-2}$.
                Then given any $q_*\in(0,1)\cap\mathbb{Q}$, choose positive integers $a$ and $b$ such that $q_*=\frac{a+4}{b-2}$.
        
                As $S$ is a power dominating set and $F$ is a failed power dominating set,  $\mathcal{E}\left(G_{a,b};F,0\right) < \mathcal{E}\left(G_{a,b};S,0\right)$.
                Then for all $0\leq q\leq \frac{a+4}{b-2}$, it follows that $\mathcal{E}\left(G_{a,b};F,q\right)\leq\mathcal{E}\left(G_{a,b};S,q\right)$.
                Since $g(q)$ is a quadratic whose roots are $\frac{a+4}{b-2}$ and $1$, it necessarily follows that $\mathcal{E}\left(G_{a,b};S,q\right)\leq\mathcal{E}\left(G_{a,b};F,q\right)$ for $\frac{a+4}{b-2}\leq q\leq 1$.
            \end{proof}
    \section{Probability of observing the entire graph}\label{sec:probability}    
        Given a graph $G$, an initial placement of PMUs $S$, and a probability $q$ of PMU failure, we now investigate the probability that the entire vertex set will be observed at the termination of the fragile power domination process.
        \cref{cor:wholethingbestthing} shows that when restricting to sets as PMU placements, the best probability of observing the entire graph $G$ occurs when $S=V(G)$.
        We examine this in \cref{sec:entire}.
        We then consider how the probability of observing the entire graph is related to the failed power domination number and the $k$-PMU-defect-robust power domination number in \cref{sec:entire-other}. 
        We conclude with examples and calculate the exact probability for stars in \cref{sec:entire-ex}.
        \subsection{Using the entire the vertex set}\label{sec:entire} 
            If a PMU is placed on every vertex, the probability that this placement observes the entire graph is connected to the total number of power dominating sets.
            In \cite{bpst23}, Brimkov, et. al. introduced the \emph{power domination polynomial} as a tool to count the number of power dominating sets of a graph. 
            \begin{defn}\label{def:pdpolynomial}
                Given a graph $G$ on $n$ vertices, let $p(G;i)$ be the number of power dominating sets of size $i$. Define the power domination polynomial as
                \[ \mathcal{P}(G;x)=\sum_{i=1}^{n} p(G;i)x^i.\]
            \end{defn}
            Using power domination polynomial notation, the probability of observing the entire graph with a PMU placed on every vertex is as follows.
            \begin{obs}\label{thm:proballnobservedexact}
                Let $G$ be a graph on $n$ vertices and $q$ be a probability of PMU failure.
                Then,
               \[ \operatorname{Prob}\Big[|\operatorname{Obs}\left(G;V(G),q\right)|=n\Big] = \sum_{i=1}^{n} p(G;i)q^{n-i}(1-q)^i. \]
            \end{obs}
            If $i>\overline{\gamma}_P\left(G\right)$ then all subsets of $V(G)$ of with $i$ vertices are power dominating sets.
            This gives a lower bound for the probability that $V(G)$ observes the entire graph.
            \begin{prop}
                Let $G$ be a graph on $n$ vertices with $\overline{\gamma}_P\left(G\right)=f$ and let $q$ be a probability of PMU failure.
                Then
                \[ \operatorname{Prob}\Big[|\operatorname{Obs}\left(G;V(G),q\right)|=n\Big] \geq \sum_{i=f+1}^{n} \binom{n}{i}q^{n-i}(1-q)^i. \]
            \end{prop}
            For a connected graph on at least $3$ vertices, $p(G;i)=\binom{n}{i}$ for $n-2 \leq i \leq n$ \cite[Corollary 4]{bpst23}.
            Moreover, any such graph has $\overline{\gamma}_P\left(G\right) \leq n-3$. This gives the following lower bound.
            \begin{cor}\label{cor:nminus2subsets}
                Let $G$ be a connected graph on $n \geq 3$ vertices and let $q$ be a probability of PMU failure.
                Then 
                \[ \operatorname{Prob}\Big[|\operatorname{Obs}\left(G;V(G),q\right)|=n\Big]\geq \binom{n}{n-2}q^2(1-q)^{n-2}+nq(1-q)^{n-1}+(1-q)^{n}. \]
            \end{cor}
        \subsection{Relating to other power domination parameters}\label{sec:entire-other}
            In this section we produce estimates for observing the entire graph using the failed power domination number and the $k$-PMU-defect-robust power domination number.
            \begin{prop} \label{prop:FailedProb}
                Let $G$ be a graph on $n$ vertices, $S$ be a set of vertices, and $q\in(0,1)$ be a probability of PMU failure.
                If $|S|\geq \overline{\gamma}_P\left(G\right)+(\log\varepsilon)/(\log q)$, then $$\operatorname{Prob}\big[|\operatorname{Obs}\left(G;S,q\right)|\big]=n\geq 1-\varepsilon.$$
            \end{prop}
            \begin{proof}
                Let $G$ be a graph and $q\in(0,1)$.
                Say $\overline{\gamma}_P\left(G\right)=f$. Then
                \begin{align*}
                    \operatorname{Prob} \big[ |\operatorname{Obs}\left(G;S,q\right)|=n \big] \geq\operatorname{Prob}\big[|S^*|\geq f+1\big] 
                    &=1-\operatorname{Prob} \big[|S^*|\leq f\big] \\
                    &\geq 1-q^{|S|-f}
                \end{align*}                        
                since
                \[\operatorname{Prob} \big[|S^*|\leq f\big]=\sum_{i=0}^f\operatorname{Prob} \big[|S^*|=i\big]\geq\operatorname{Prob}\big[|S^*|=f\big]=\operatorname{Prob}\big[|S\setminus S^*|=|S|-f\big]=q^{|S|-f}\]
                and so $\operatorname{Prob} \big[ |\operatorname{Obs}\left(G;S,q\right)|=n \big]\geq 1-\varepsilon$ is implied by $1-q^{|S|-f}\geq1-\varepsilon$ which is equivalent to
                \[\varepsilon\geq q^{|S|-f}\iff \log\varepsilon\geq(|S|-f)\log q\iff \frac{\log \varepsilon}{\log q}+f\leq |S|. \qedhere\]
            \end{proof}
            \cref{prop:FailedProb} is most useful for graphs where $\gamma_P\left(G\right)$ and $\overline{\gamma}_P\left(G\right)$ are relatively close, such as $K_{3,3,\ldots,3}.$
            However, there are graphs with an arbitrarily large gap between $\gamma_P\left(G\right)$ and $\overline{\gamma}_P\left(G\right)$, such as the star $S_n$.
            For these graphs, \cref{prop:FailedProb} may not give a useful result, such as implying $|S|> n$ when $S$ must be a set.
            When $\overline{\gamma}_P\left(G\right)=0$, \cref{prop:FailedProb} yields the following.
            \begin{cor}\label{cor:gpf0observeall} 
            Let $G$ be a graph on $n$ vertices, $S$ a set of vertices, and $q\in(0,1)$ a probability of PMU failure.
            If $\overline{\gamma}_P\left(G\right)=0$, then
                \begin{align*}
                    \operatorname{Prob} \big[ |\operatorname{Obs}\left(G;S,q\right)|=n \big]=1-q^{|S|}.
                \end{align*} 
                Equivalently, if $\overline{\gamma}_P\left(G\right)=0$ and $|S|\geq(\log\varepsilon)/(\log q)$, then $$\operatorname{Prob} \big[ |\operatorname{Obs}\left(G;S,q\right)|=n \big]\geq 1-\varepsilon.$$
            \end{cor}
            \cref{cor:gpf0observeall} applies to paths, cycles, wheels, and complete graphs.
            For more examples of graphs with failed power domination number zero, see \cite{gjlr20}.
            
            A higher probability of observing the entire graph is related to PMU-defect-robust power domination.
            A $k$-rPDS is a good choice for a PMU placement because any $k$ PMUs can fail and not compromise the observability of the graph.
            \begin{thm}\label{proballnrobust}
                Given a graph $G$ on $n$ vertices, a $k$-PMU-defect-robust power dominating set or multiset $S$, and a probability of PMU failure $q$, the probability that $S$ observes the entire graph $G$ is
                \[ \operatorname{Prob}\Big[|\operatorname{Obs}\left(G;S,q\right)|=n\Big] \geq 1-\sum_{i=0}^{|S|-k-1} \binom{|S|}{i} q^{|S|-i}(1-q)^{i}.
                \]
                Equality holds when $|S|-k=\gamma_P\left(G\right)$.
            \end{thm}
            \begin{proof}
                Let $S$ be a $k$-rPDS, then all subsets or submultisets $W \subseteq S$ with $|W| \geq |S|-k$ are power dominating sets.
                The probability that $S$ observes all of $G$ is at least the probability that $|W| \geq |S|-k$, or one minus the probability that $|W|<|S|-k$.
                The inequality then follows.
                
                Note that if $|S|-k = \gamma_P\left(G\right)$, then $W\subseteq S$ is a power dominating set if and only if $|W|\geq |S|-k$ and so equality holds.
            \end{proof}
        \subsection{Examples}\label{sec:entire-ex}
            Consider the complete bipartite graph $K_{3,3}$, a probability of PMU failure of $q=0.1$, and a desired probability of observing the entire graph of $95\%$.
            Note that any set of $2$ distinct vertices of $K_{3,3}$ forms a power dominating set.
            For any $2$-set $S$, use \cref{proballnrobust} to obtain
            \[ \operatorname{Prob}\Big[|\operatorname{Obs}\left(K_{3,3};S,0.1\right)|=6\Big] = 1-\sum_{i=0}^{1} \binom{2}{i} (0.1)^{2-i}(0.9)^{i} = 0.81, \]
            which is too small.
            If we instead utilize $S'$ consisting of any 3 distinct vertices, $S'$ is a $1$-rPDS and we find by \cref{proballnrobust} that
            \[ \operatorname{Prob}\Big[|\operatorname{Obs}\left(K_{3,3};S',0.1\right)|=6\Big] = 1-\sum_{i=0}^{1} \binom{3}{i} (0.1)^{3-i}(0.9)^{i} = 0.972.
            \]
            Hence, a 1-rPDS is sufficient for observing all of $K_{3,3}$ with the desired $95\%$ probability.
            
            We now determine the probability of observing the entire star $S_n$ given \emph{any} PMU placement. 
            \begin{prop}\label{prop:startotalobs}
                Let $S_n$ denote the star with universal vertex $v_0$ and let $q$ be a probability of PMU failure. Then for any set $S\subseteq V(S_n)$, $\operatorname{Prob} \big[ |\operatorname{Obs}\left(S_n;S,q\right)|=n \big]$ is equal to
                \begin{align*}
                    \begin{cases}
                        (1-q)^{|S|},&\text{if $v_0\not\in S$ and $|S|= n-2$}\\
                        (1-q)^{|S|}+|S|q(1-q)^{|S|-1},&\text{if $v_0\not\in S$ and $|S|= n-1$}\\
                        1-q,&\text{if $v_0\in S$ and $|S|\leq n-2$}\\
                        1-q+q(1-q)^{|S|-1},&\text{if $v_0\in S$ and $|S|= n-1$}\\
                        1-q+q(1-q)^{|S|-1}+(|S|-1)q^2(1-q)^{|S|-2},&\text{if $v_0\in S$ and $|S|= n$}\\
                        0,&\text{otherwise.} 
                    \end{cases}
                \end{align*}
            \end{prop}
            \begin{proof}
                First suppose $v_0\not\in S$.
                If $S$ does not have at least $n-2$ leaves, then $S_n$ can never be fully observed.
                If $|S|=n-2$, then we need every vertex to fully observe $S_n$, which occurs with probability $(1-q)^{|S|}$.
                If $|S|=n-1$, then to fully observe $S_n$ either no PMUs fail or one PMU fails.
                These are disjoint events which occur with probabilities $(1-q)^{|S|}$ and $|S|q(1-q)^{|S|-1}$ respectively.
                
                Now suppose $v_0\in S$.
                If $|S|\leq n-2$ then the only way that $S_n$ is fully observed is if $v_0\in S^*$, which occurs with probability $1-q$.
                If $|S|=n-1$, the probability that $S_n$ is fully observed is equal to
                \[\operatorname{Prob} \big[ v_0\in S^*\vee\{v_0\not\in S^*\wedge |S^*|=n-2\} \big]=(1-q)+q(1-q)^{|S|-1}.\]
                Finally, if $|S|=n$ then the probability that $S_n$ is fully observed is equal to
                \begin{align*}
                    \operatorname{Prob} \big[v_0\in S^*\vee\{v_0\not\in S^*\wedge |S^*|\geq n-2\}\big]
                    &=(1-q)+q \operatorname{Prob} \big[ \big(|S^*|=n-2\vee |S^*|=n-1\big):v_0\not\in S^*\big]\\
                    &=(1-q)+q\left((|S|-1)q(1-q)^{|S|-2}+(1-q)^{|S|-1}\right)
                \end{align*}
                Therefore the result holds.
            \end{proof}
            As a secondary example, consider the star $S_{18}$, a probability of PMU failure of $q=0.1$, and a desired probability of observing the entire graph of $95\%$.
            Then by \cref{prop:startotalobs} we determine that utilizing the entire vertex set yields
            \[\operatorname{Prob}\Big[|\operatorname{Obs}\left(S_{18};V(S_{18}),q\right)|= 18 \Big] = 1-(0.1)+(0.1)(0.9)^{17}+(17)(0.1)^2(0.9)^{16} = 0.9482,\]
            which means that it is not possible to utilize \emph{sets} of vertices in order to obtain the desired probability of observing the entire graph.
            However, if we instead place two PMUs on the universal vertex $v_0$ to create a $1$-rPDS,
            \[\operatorname{Prob}\Big[|\operatorname{Obs}\left(S_{18};\{v_0,v_0\},q\right)|= 18 \Big] = 1-(0.1)^2 = 0.99.\]
            Hence, it is only with \emph{multisets} that one can achieve $95\%$ probability to observe $S_{18}$.
    \section{The expected value polynomial for graph families}\label{sec:families}
        We determine the expected value polynomial for a PMU placement for the following graph families: a generalization of barbell graphs, stars, and complete multipartite graphs.
        \subsection{Generalized Barbells}
            Given a graph $G$, let $\overline{G}$ denote the \emph{complement of $G$} where $V(\overline{G})=V(G)$ and $E(\overline{G})=\{xy:xy\not\in E(G)\}$.
            \begin{lem}\label{lem:barbell}
                Let $G\in\{C_n,W_n,K_n,\overline{C_{n+2}} : n\geq3\}$.
                Construct $H$ from $G$ by attaching a leaf to an arbitrary vertex of $G$.
                Then for any $v\in V(G)$, $\{v\}$ is a power dominating set for $H$.
            \end{lem}
            \begin{proof}
                Throughout what follows, let $x$ be the leaf attached to some vertex of $G$.
                
                If $G=K_n$ then any vertex $v\in V(G)$ dominates $V(G)$ and observes $x$ in either the domination step or at most one zero forcing step.
                
                If $G=C_n$, then any $v\in V(G)$ is a power dominating set for $C_n$, and the power domination process is unaffected by attaching $x$. 
                Then $x$ becomes observed after at most one additional zero forcing step. 
                
                When $G=W_n$, the leaf $x$ can be adjacent to either the cycle or the universal vertex.
                If $x$ is adjacent to the cycle, then for any $v\in V(G)$ the center of the wheel is observed in domination step, after which the $G=C_n$ case is recovered.
                If $x$ is adjacent to the universal vertex, power domination proceeds as on $W_n$ before observing $x$ in at most one additional step.
                
                Finally, let $G=\overline{C_{n+2}}$, label the vertices of $G$ cyclically $v_0,\ldots,v_{n-1}$ with indices modulo $n$, and by symmetry assume $x$ is adjacent to $v_0$ in $H$.
                Let $v_i\in V(G)$ be arbitrary.
                If $i=0$ then we recover the case of $\overline{\gamma}_P\left(\overline{C_{n+2}}\right)$ from \cite{gjlr20}.
                Otherwise, notice that $v_i$ dominates everything except $v_{i-1}$ and $v_{i+1}$.
                Then there are three unobserved vertices remaining.
                At least one of $v_{i-2}$ or $v_{i+2}$ is not adjacent to $x$, without loss of generality say $v_{i+2}$ is not adjacent to $x$.
                Then $v_{i+2}$ observes $v_{i-1}$ as its only unobserved neighbor.
                Finally, one of $v_{i-1}$ or $v_{i-2}$ is not a neighbor of $x$ and hence observes $v_{i+1}$.
                Since all of $V(G)$ has been observed, $x$ becomes observed.
            \end{proof}
            We now introduce the generalized barbell graph.
            \begin{defn}
                Let $G_1, G_2$ be graphs and pick $x_1\in V(G_1)$, $x_2 \in V(G_2)$.
                Add the edge $x_1x_2$ and subdivide it $m\geq 0$ times.
                The resultant graph is a \emph{generalized barbell graph}, denoted by $B(G_1,x_1, G_2, x_2, m)$.
            \end{defn}
            \begin{figure}[ht]
                \centering
                \includegraphics{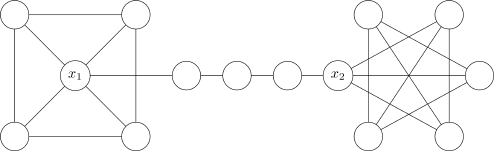}
                \caption{The generalized barbell graph $B(W_5,x_1,\overline{C_6},x_2,3)$.}
                \label{fig:genbarbellexample}
            \end{figure}
    
            The $m$ vertices of a generalized barbell graph $B(G_1,x_1,G_2,x_2,m)$ between $x_1$ and $x_2$ are referred to as the \emph{central path} of the graph.
            If $m=0$, then $G_1$ and $G_2$ are connected by an edge, and the central path is empty.
            Note that $B(K_{n},x_1,K_{n},x_2,0)$ for any $x_1,x_2$ is the usual barbell graph.
            \cref{fig:genbarbellexample} demonstrates the generalized barbell graph $G(W_5,x_1,\overline{C_6},x_2,3)$.
            We now determine the expected value polynomial for certain generalized barbell graphs.
            \begin{prop}\label{prop:barbell}
                Let $G_1, G_2 \in \{K_1,K_n, W_n, C_n,\overline{C_{n+2} } : n\geq 3\}$ with $|V(G_1)|=\ell$ and $|V(G_2)|=n$, and construct the generalized barbell graph $G=B(G_1,x_1,G_2,x_2,m)$ for any $x_i \in V(G_i)$.
                Let $S_{r,s,t}$ be a subset or submultiset of $V(G)$ containing $r$ vertices from $G_1$, $s$ vertices from the central path, and $t$ vertices from $G_2$.
                Note if $m=0$ then $s=0$ necessarily.
                Then for a given probability of PMU failure $q$,
                \(\mathcal{E}\left(G;S_{r,s,t},q\right)\) is given by
                \[(\ell+m+n)(1-q^r)(1-q^t)+(\ell+m+1)(1-q^r)q^t+(m+n+1)q^r(1-q^t)+(m+2)q^{r+t}(1-q^s).\]
            \end{prop}
            \begin{proof}
                Partition $S_{r,s,t}=R\sqcup S\sqcup T$ as disjoint vertex subsets of size $r,s,$ and $t$ corresponding to $G_1$, the central path, and $G_2$ respectively.
                Denote by $R^*$ the random set which contain a vertex $v\in R$ with probability $1-q$, and similarly write $S^*$ and $T^*$.
                Notice that if $|R^*|\geq1$ then by \cref{lem:barbell}, $G_1$, the central path, and $x_2$ will be observed.
                Similarly, if $|T^*|\geq1$ then $G_2$, the central path, and $x_1$ will be observed.
                
                On the other hand, if $|R^*|=|T^*|=0$ but $|S^*|>0$, then the entire central path, $x_1$, and $x_2$ will be observed.
                Hence, we can expand $\mathcal{E}\left(G;S_{r,s,t},q\right)$ using the definition of expected value to obtain
                \begin{align*}
                &(\ell+m+n)\operatorname{Prob} \big[\exists v\in R^*\wedge\exists v\in T^*\big]+(\ell+m+1)\operatorname{Prob} \big[\exists v\in R^*\wedge\not\exists v\in T^*\big]+\\
                &(m+n+1)\operatorname{Prob} \big[\not\exists v\in R^*\wedge\exists v\in T^*\big]+(m+2)\operatorname{Prob} \big[\not\exists v\in R^*\wedge\not\exists v\in T^*\wedge \exists v\in S^*\big].
                \end{align*}
                The result follows from some basic probability calculations.
            \end{proof}
        \subsection{Stars}
        To calculate the expected number of observed vertices in the star, we first break down the possible types of PMU placements into lemmas.
            \begin{lem}\label{lem:StarNoCentral}
                Let $S$ be a subset of the vertices of $S_n$ which does \textbf{not} contain the universal vertex with $|S|\leq n-3$, and let $q$ be a probability of PMU failure.
                Then 
                \begin{align*}
                    \mathcal{E}\left(S_n;S,q\right)=1+|S|(1-q)-q^n.
                \end{align*}
            \end{lem}
            \begin{proof}
                Since $|S|\leq n-3$ and $S$ does not contain the universal vertex, for all non-empty $W\subseteq S$ we have that $|\operatorname{Obs}\left(S_n;W\right)|=|W|+1$. Hence
                \begin{align*}
                    \mathcal{E}\left(S_n;W,q\right)=\sum_{i=1}^{|S|}(i+1)\binom{|S|}{i}q^{|S|-i}(1-q)^i &=\sum_{i=0}^{|S|}\left((i+1)\binom{|S|}{i}q^{|S|-i}(1-q)^i\right)-q^{|S|}\\
                    &=1+|S|(1-q)-q^{|S|}
                \end{align*}
                by \cref{fact1} and \cref{fact2}.
            \end{proof}
            \begin{lem}\label{lem:StarCentral}
            Let $S$ be a subset of the vertices of $S_n$ with universal vertex $v_0\in S$ and $|S|\leq n-2$, and let $q$ be a probability of PMU failure.
            Then
            \begin{align*}
                \mathcal{E}\left(S_n;S,q\right)=n(1-q)+(|S|-1)q(1-q)+q-q^{|S|}.
            \end{align*}
            \end{lem}
            \begin{proof}
                Let $|S|=k\leq n-2$ and let $v_0$ denote the universal vertex.
                The entire graph is observed when $v_0\in S^*$, which occurs with probability $(1-q)$ and so
                \begin{align*}
                    \mathcal{E}\left(S_n;S,q\right)
                    &=\sum_{\substack{W\subseteq S:\\v_0\in W}}|\operatorname{Obs}\left(S_n;W\right)|\,q^{|S\setminus W|}(1-q)^{|W|}
                    +\sum_{\substack{W\subseteq S:\\v_0\not\in W}}|\operatorname{Obs}\left(S_n;W\right)|\,q^{|S\setminus W|}(1-q)^{|W|}\\
                    &=(1-q)\sum_{i=0}^{k-1}n\binom{k-1}{i}q^{k-1-i}(1-q)^{i}+q\sum_{i=1}^{k-1}(i+1)\binom{k-1}{i}q^{k-1-i}(1-q)^i\\
                    &= n(1-q)\sum_{i=0}^{k-1}\binom{k-1}{i}q^{k-1-i}(1-q)^i\\
                    &~~~~+q\sum_{i=0}^{k-1}i\binom{k-1}{i}q^{k-1-i}(1-q)^i+q\sum_{i=0}^{k-1}\left(\binom{k-1}{i}q^{k-1-i}(1-q)^i\right)-q^k\\
                    &=n(1-q)+(k-1)(1-q)q+q-q^k
                \end{align*}
                where on the second line, the first sum is over the number of vertices belonging to a set $W$ containing $v_0$, and the second sum is over vertices belonging to a set $W$ not containing $v_0$.
                Finally, applying \cref{fact1} and \cref{fact2}, the result follows.
            \end{proof}
            We can now determine the expected value polynomial for stars.
            \begin{thm}
                Let $S_n$ denote the star on $n$ vertices, and let $v_0$ denote the universal vertex.
                If $S$ is a set of vertices and $q$ is a probability of PMU failure, then $\mathcal{E}\left(S_n;S,q\right)$ is equal to
                \begin{align*}
                    \begin{cases}
                        1+|S|(1-q)-q^{|S|},&\text{if $v_0\not\in S$ and $|S|\leq n-3$}\\
                        1+|S|(1-q)-q^{|S|}+(1-q)^{|S|},&\text{if $v_0\not\in S$ and $|S|= n-2$}\\
                        1+|S|(1-q)-q^{|S|}+|S|q(1-q)^{|S|-1},&\text{if $v_0\not\in S$ and $|S|= n-1$}\\
                        n(1-q)+(|S|-1)q(1-q)+q-q^{|S|},&\text{if $v_0\in S$ and $|S|\leq n-2$}\\
                        n(1-q)+(|S|-1)q(1-q)+q-q^{|S|}+q(1-q)^{|S|-1},&\text{if $v_0\in S$ and $|S|= n-1$}\\
                        n(1-q)+(|S|-1)q(1-q)+q-q^{|S|}+(|S|-1)q^2(1-q)^{|S|-2},&\text{if $v_0\in S$ and $|S|= n$}
                    \end{cases}
                \end{align*}
            \end{thm}
            \begin{proof}
                Suppose first $v_0\not\in S$. Then $|S|\leq n-3$ is \cref{lem:StarNoCentral}.
                Notice in the proof of \cref{lem:StarNoCentral} that we assume $S^*$ always observes $|S^*|+1$ vertices.
                However, if $|S|=n-2$ and $S^*=S$ with probability $(1-q)^{|S|}$, then actually $|S^*|+2$ vertices are observed by way of the remaining leaf not accounted for in \cref{lem:StarNoCentral}.
                Thus by linearity of expectation, the result of \cref{lem:StarNoCentral} increases by $(1-q)^{|S|}$.
                When $|S|=n-1$, there are $\binom{|S|}{|S|-1}=|S|$ sets $S^*$ containing all but one leaf, each of which observes $|S^*|+2$ vertices.
                This occurs with probability $q(1-q)^{|S|-1}$ and so increases the result of \cref{lem:StarNoCentral} by $|S|q(1-q)^{|S|-1}$. 
                
                Now suppose $v_0\in S$.
                The case of $|S|\leq n-2$ is \cref{lem:StarCentral}.
                In the proof of \cref{lem:StarCentral} we assume that if $v_0\not\in S^*$, then we observe $|S^*|+1$ vertices.
                When $|S|=n-1$, $S$ contains $v_0$ and all but one leaf.
                If $S^*=S\setminus\{v_0\}$ then $S^*$ observes $|S^*|+2$ vertices, or one vertex not accounted for by \cref{lem:StarCentral}.
                This occurs with probability $q(1-q)^{|S|-1}$ and so add this to the result of \cref{lem:StarCentral} using linearity of expectation.
                When $|S|=n$, we must adjust for the case when $v_0$ and a leaf fail.
                There are $|S|-1$ ways for a single leaf to fail, and each of these occurs with probability $q^2(1-q)^{|S|-2}$.
                Thus we add $(|S|-1)q^2(1-q)^{|S|-2}$ to the output of \cref{lem:StarCentral} to get our result.
            \end{proof}
        \subsection{Complete Multipartite Graphs}
            We now calculate the expected value polynomial for complete multipartite graphs with parts of size at least $2$.
            \begin{lem}\label{lem:TwoParts}
                Let $G=K_{r_1,\ldots,r_k}$ with $k\geq 2$ and $r_i\geq 2$ for all $i$, let $S=S_{\ell_1,\ldots,\ell_k}\subseteq V(G)$ be a set such that $S$ contains $\ell_i$ vertices from the $i$th partition of $G$, and let $q$ be a probability of PMU failure.
                Then the probability that $S^*$ contains vertices from two distinct parts of $G$ is
               \begin{align*}
                   1-q^\ell\left(1+\sum_{i=1}^k (q^{-\ell_i}-1)\right)
               \end{align*}
            \end{lem}
            \begin{proof}
                Let $G=K_{r_1,\ldots,r_k}$, the sets $A_1, \ldots, A_k$ be the disjoint partitions of $V(G)$, and the set $S=S_{\ell_1,\ldots,\ell_k}\subseteq V(G)$ be such that $S$ contains $\ell_i$ vertices from $A_i$.
                Let us frame the problem as its complement:
                \begin{align*}
                    \operatorname{Prob} \big[\exists v,w\in S^*\text{~s.t.~}v\in A_i,w\in A_j\big]=1-\operatorname{Prob} \big[S^*\subseteq A_i\text{~for some $i$}\big].
                \end{align*}
                Then using the property of inclusion and exclusion,
                \begin{align*}
                    \operatorname{Prob} \big[S^*\subseteq A_i\text{~for some $i$}\big]
                    &=\operatorname{Prob}\left[\bigvee_{i=1}^k \big\{S^*\subseteq A_i\big\}\right]\\
                    &=\sum_{I\subseteq [k]}(-1)^{|I|+1}\operatorname{Prob}\left[\bigwedge_{i\in I}\big\{S^*\subseteq A_i\big\}\right]\\
                    &=\sum_{i=1}^k\operatorname{Prob} \big[S^*\subseteq A_i\big]+\sum_{\substack{I\subseteq[k]\\|I|\geq 2}}(-1)^{|I|+1}\operatorname{Prob} \big[S^*=\varnothing\big]\\
                    &=\sum_{i=1}^k q^{|S\setminus(S\cap A_i)|}-\sum_{i=2}^k(-1)^{i}\binom{k}{i}q^{|S|}.
                \end{align*}
                since $q^{|S\setminus(S\cap A_i)|}$ is the probability that $S^*$ is contained in $A_i$ and $\operatorname{Prob} \big[S^*=\varnothing\big]=q^{|S|}$.
                We now use \cref{fact3} and substitute $|S|=\ell$, $|S\cap A_i|=\ell_i$ to find 
                \begin{align*}
                    \sum_{i=1}^k q^{|S\setminus(S\cap A_i)|} -\sum_{i=2}^k(-1)^{i}\binom{k}{i}q^{|S|}
                    &=\sum_{i=1}^k q^{\ell-\ell_i}-q^\ell\left(\sum_{i=0}^k(-1)^i\binom{k}{i}-(-k)-1\right)\\
                    &=q^\ell\left(1+\sum_{i=1}^k\left(q^{-\ell_i}-1\right)\right)
                \end{align*}
                from which the result follows.
            \end{proof}
            \begin{thm}
                Let $G=K_{r_1,\ldots,r_k}$ with $k\geq 2$ and $r_i\geq 2$ for all $i$, let $S=S_{\ell_1,\ldots,\ell_k}\subseteq V(G)$ be a set such that $S$ contains $\ell_i$ vertices from the $i$th partition of $G$, and let $q$ be a probability of PMU failure.
                Denote $r=|V(G)|$ and $\ell=|S|$.
                Then $\mathcal{E}\left(G;S,q\right)$ is given by
                \begin{align*}
                    r\left(1-q^\ell\left(1+\sum_{i=1}^k (q^{-\ell_i}-1)\right)\right)+\sum_{i=1}^k q^{\ell-\ell_i}(1-q^{\ell_i})\Big(r+\ell_i q^{\ell_i}-q^{\ell_i}r-\ell_i q+\ell_i q(1-q)^{\ell_i-1}\Big)
                \end{align*}
            \end{thm}
            \begin{proof}
                Let $G$, $S$, $q$, $r$, and $\ell$ be as in the theorem statement, and let $A_i$ denote the vertex set of the $i$th partition.
                Then $\mathcal{E}\left(G;S,q\right)$ can be written as
                \begin{align*}
                    &r\operatorname{Prob} \big[\exists v,w\in S^*\text{~s.t.~}v\in A_i,w\in A_j\big]\\
                    &+\operatorname{Prob} \big[\varnothing\neq S^*\subseteq A_1\big]\sum_{\varnothing\neq W\subseteq A_1}|\operatorname{Obs}\left(G;W\right)|q^{|A_1\setminus W|}(1-q)^{|W|}\\
                    &~\vdots\\
                    &+\operatorname{Prob} \big[\varnothing\neq S^*\subseteq A_k\big]\sum_{\varnothing\neq W\subseteq A_k}|\operatorname{Obs}\left(G;W\right)|q^{|A_k\setminus W|}(1-q)^{|W|}.
                \end{align*}
                Notice that for a fixed $i$, 
                \begin{align*}
                    \operatorname{Prob} \big[\varnothing\neq S^*\subseteq A_i\big]=\operatorname{Prob} \big[S^*\neq\varnothing : S^*\subseteq A_i\big]\operatorname{Prob} \big[S^*\subseteq A_i\big]=(1-q^{\ell_i})q^{\ell-\ell_i}.
                \end{align*}
                Fixing $t\in[k]$, we turn our attention to the sum
                \begin{align}\label{eqn:blar}
                    \sum_{\varnothing\neq W\subseteq A_t}|\operatorname{Obs}\left(G;W\right)|q^{|A_t\setminus W|}(1-q)^{|W|}.
                \end{align}
                When $|W|=\ell_t-1$, notice $|\operatorname{Obs}\left(G;W\right)|=r$ instead of $r-1$ because the vertices in other parts will be able to observe their single white neighbor in $A_t$.
                Hence we can write \cref{eqn:blar} as
                \begin{align*}
                  \sum_{i=1}^{\ell_t}\left(\binom{\ell_t}{i}(r-\ell_t+i)q^{\ell_t-i}(1-q)^i\right)-\binom{\ell_t}{\ell_t -1}(r-\ell_t+\ell_t-1)q(1-q)^{\ell_t -1}+\binom{\ell_t}{\ell_t -1}rq(1-q)^{\ell_t -1}.
                \end{align*}
                Introducing an $i=0$ term we get
                \begin{align*}
                 \nonumber&\sum_{i=0}^{\ell_t}\left(\binom{\ell_t}{i}(r-\ell_t+i)q^{\ell_t-i}(1-q)^i\right)
                    -(r-\ell_t)q^{\ell_t}\\
                    & -\ell_t(r-\ell_t+\ell_t-1)q(1-q)^{\ell_t -1}
                    +\ell_t rq(1-q)^{\ell_t -1}
                \end{align*}
                and distributing over $(r-\ell_t+i)$ gives
                \begin{align*}
                  \nonumber&(r-\ell_t)\sum_{i=0}^{\ell_t}\binom{\ell_t}{i}q^{\ell_t-i}(1-q)^i
                    +\sum_{i=0}^{\ell_t}\binom{\ell_t}{i}i q^{\ell_t-i}(1-q)^i\\
                    & -(r-\ell_t)q^{\ell_t}
                    -\ell_t(r-\ell_t+\ell_t-1)q(1-q)^{\ell_t -1}
                    +\ell_t rq(1-q)^{\ell_t -1}.
                \end{align*}
                We now apply the binomial theorem to get
                \begin{align*}
                    r-\ell_t+\ell_t(1-q)-(r-\ell_t)q^{\ell_t}
                    -\ell_t(r-\ell_t+\ell_t-1)q(1-q)^{\ell_t -1}
                    +\ell_t rq(1-q)^{\ell_t -1}
                \end{align*}
                and after algebraic simplification we are left with 
                \begin{align*}
                    r+\ell_t q^{\ell_t}-q^{\ell_t}r-\ell_t q+\ell_t q(1-q)^{\ell_t-1}.
                \end{align*}
                Finally, recall that $\mathcal{E}\left(G;S,q\right)$ is given by
                \begin{align*}
                    \nonumber & r\operatorname{Prob} \big[\exists v,w\in S^*\text{~s.t.~}v\in A_i,w\in A_j\big]\\
                    &+\sum_{i=1}^k\left[\operatorname{Prob} \big[\varnothing\neq S^*\subseteq A_i\big]\sum_{\varnothing\neq W\subseteq A_i}|\operatorname{Obs}\left(G;W\right)|q^{|A_i\setminus W|}(1-q)^{|W|}\right]
                \end{align*}
                and so substituting terms from above and \cref{lem:TwoParts} gives the result.
            \end{proof}
    \section{Future Work}\label{sec:futurework}
        In this paper, we introduced the concept of power domination with random sensor failure and tools to study the fragile power domination process.
        Many of our results depend on particular PMU placements or specific graphs.
        We demonstrated connections to power domination variations such as failed power domination, fault-tolerant power domination, and PMU-defect-robust power domination.
        Are there more general structural conditions or other graph parameters that could be used to study fragile power domination?
        
        In \cite{hhhh02}, it is shown one can always find a power dominating set with vertices all having degree three or higher.
        \cref{fig:spectrumplot} suggests that, when choosing PMU placements, high degree vertices may be a better choice.
        Particularly, the set $\{4,4\}$ containing two vertices of degree 4 observes more vertices than the set $\{2,5\}$ containing a leaf and a degree 4 vertex for some values of $q$.
        Is this a behavior that can be characterized? 
        Is there a relationship between the degrees of vertices in a PMU placement and the expected value polynomial? 
        An easy result to see is that if a single PMU is placed at the vertex of lowest degree, we obtain $\mathcal{E}\left(G;S,q\right) \geq (\delta(G) +1 ) (1-q)$ as only the closed neighborhood of said vertex is observed.
        
        In \cref{thm:fishie}, we demonstrated a family of graphs for which the intersection of the expected value polynomials for failed power dominating sets and power dominating sets can occur at any rational probability.
        What can be said about the graph structures that realize this phenomenon?
        
        We utilized the power domination polynomial from \cite{bpst23} to study the probability of observing the entire graph in \cref{sec:probability}.
        There are also results on graph products presented in \cite{bpst23} that might be extendable to fragile power domination.
    \section*{Acknowledgements}
        This project was sponsored, in part, by the Air Force Research Laboratory via the Autonomy Technology Research Center. 
        This research was also supported by Air Force Office of Scientific Research award 23RYCOR004.

    \bibliographystyle{plain}
    \bibliography{bib}
\end{document}